\documentclass{amsart}
\usepackage{tikz}
\usepackage{setspace}
\newtheorem{theorem}{Theorem}[section]
\newtheorem{lemma}[theorem]{Lemma}
\theoremstyle{definition}
\newtheorem{definition}[theorem]{Definition}
\newtheorem*{conj}{Conjecture}
\theoremstyle{remark}

\numberwithin{equation}{section}
\newcommand{\Z}{\mathbb{Z}}
\newcommand{\Q}{\mathbb{Q}}
\newcommand{\R}{\mathbb{R}}

\newcommand{\OLU}{{\mathcal{O}_L^*}}
\newcommand{\Olu}{{\mathcal{O}_l^*}}
\newcommand{\abs}[1]{\lvert#1\rvert}

\newcommand{\norm}[2]{\lVert#1\rVert_{#2}}
\newcommand{\LOG}{\operatorname{LOG}}
\newcommand{\Gal}{\operatorname{Gal}}
\newcommand{\lof}{\log\left(\tfrac{1+\sqrt{5}}{2}\right)}
\newcommand{\lofs}{\log^2\left(\tfrac{1+\sqrt{5}}{2}\right)}
\allowdisplaybreaks

\begin{document}

\title[A real quartic Galois case of the B--RV conjecture]
{Bertrand's and Rodriguez Villegas' Conjecture for real quartic Galois extensions of the rationals}

\author{DOHYEONG KIM}
\address{Department of Mathematical Sciences and Institute for Data Innovation in Science, Seoul National University, Gwanak-ro 1, Gwankak-gu, Seoul, South Korea 08826}
\email{dohyeongkim@snu.ac.kr}

\author{SEUNGHO SONG}
\address{Department of Mathematical Sciences, Seoul National University, Gwanak-ro 1, Gwankak-gu, Seoul, South Korea 08826}
\email{shsong0611@snu.ac.kr}

\date{\today}

\begin{abstract}
      The conjecture due to Bertrand and Rodriguez Villegas asserts that the 1-norm of the nonzero element in an exterior power of the units of a number field has a certain lower bound. 
     For the exterior square case of totally real quartic extensions of the rationals, Costa and Friedman gave a lower bound of 0.802. 
     We prove that the bound can be improved to 1.134 when the extension is further assumed to be  Galois.
\end{abstract}

\maketitle

\section{Introduction}

While Dirichlet's theorem determines the rank of the unit group in a given algebraic number field in terms of its Archimedean places, less is known about its geometry when viewed as a lattice in a Euclidean space via the associated logarithm embedding. 
The resulting lattice and its exterior powers are conjecturally subject to certain numerical constraints \cite{Be1997, CFS2022}.
We begin by introducing notation to clarify the nature of this numerical constraint.

\begin{definition}
    Let~$L$ be a number field and~$\mathcal{A}_L$ be the set of its Archimedean places. 
    Then the logarithmic embedding of the units~$\LOG:\OLU \rightarrow \R^{\mathcal{A}_L}$ into a Eucliedean space is defined by
    \begin{equation*}
        (\LOG(\gamma))_v:=e_v\log\abs{\gamma}_v 
        \text{\quad where \quad}  e_v:=\begin{cases}
        1\ &\text{if}\ v\ \text{is real},\\
        2\ &\text{if}\ v\ \text{is complex}\end{cases}
    \end{equation*}
    for~$\gamma \in \OLU$ and~$v \in \mathcal{A}_L$. 
    Here,~$\abs{\ \cdot \ }_v$ is the absolute value associated to~$v$ extending the absolute value on~$\Q$.
\end{definition}

Consider the orthonormal basis ~$\{\delta^v\}_{v \in \mathcal{A}_L}$ of~$\R^{\mathcal{A}_L}$ given by 
\begin{equation*}
    \delta_{w}^{v}:=\begin{cases}
    1\ &\text{if}\ \ w=v,\\
    0\ &\text{if}\ \ w\neq v \end{cases}
\end{equation*}
for~$w\in \mathcal{A}_L$.
Let~$\mathcal{A}_L^{[j]}$ be the set of subsets of~$\mathcal{A}_L$ having cardinality~$j$.
For each ~$I\in\mathcal{A}_L^{[j]}$, fix an ordering ~$\{v_1,v_2,...,v_j\}$ of elements of~$I$.
Define
\begin{equation*}
    \delta^{I}:=\delta^{v_1}\wedge \delta^{v_2} 
    \wedge \cdots \wedge \delta^{v_j}
\end{equation*}
to get an orthonormal basis ~$\{\delta^{I}\}_{I\in\mathcal{A}_L^{[j]}}$ of ~$\bigwedge^j\R^{\mathcal{A}_L}$.
For~$w=\sum\limits_{I \in \mathcal{A}_L^{[j]}} c_I\delta^I \in \bigwedge^j\R^{\mathcal{A}_L}$, define its 1-norm and 2-norm as
\begin{equation*}
    \norm{w}{1}:=\sum\limits_{I \in \mathcal{A}_L^{[j]}}\abs{c_I},\qquad \norm{w}{2}:=\bigg(\sum\limits_{I \in \mathcal{A}_L^{[j]}}c_I^2\bigg)^{1/2}
\end{equation*}
respectively.

Our aim in this paper is to investigate a lower bound for ~$\norm{w}{1}$ for a nonzero element ~$w$ of ~$\bigwedge^j\LOG(\OLU)$.
For a positive integer~$n$, let~$L_n$ denote the set of subfields of~$\R$ that are Galois extension of~$\Q$ of degree~$n$. 
For such fields,~$\operatorname{rank}(\OLU)=n-1$ by Dirichlet's unit theorem.
For~$1\leq j\leq n-1$, define a real number~$A_{n,j}$ as 
\begin{equation}
    A_{n,j}:=\inf_{\substack{L\in L_n \\ w\in\bigwedge^j\LOG(\OLU)\backslash\{0\}}}\{\norm{w}{1}\}.
\end{equation}
This constant is directly related to the real Galois extension case of Bertrand's and Rodriguez Villegas' conjecture.
The conjecture stated in \cite{CFS2022} is as follows.
\begin{conj}[Bertrand--Rodriguez Villegas]
    There exist two absolute constants \\~$c_0>0$ and~$c_1>1$ such that for any number field~$L$ and any
    ~$j \in \Z_{>0}$, the following inequality holds
    \begin{equation*}
        \norm{w}{1} \geq c_0c_1^j
    \end{equation*}
    for any nonzero~$w \in \bigwedge^j\LOG(\OLU) \subset \bigwedge^j\R^{\mathcal{A}_L}$.
\end{conj}

Only considering the number fields that are real Galois extensions of ~$\Q$, a consequence of the conjecture is that 
\begin{equation*}
    A_{n,j}\geq c_0c_1^j
\end{equation*}
holds, where the~$c_0,c_1$ are the constants satisfying the conjecture. 
For the specific case of ~$n=4, j=2$, the work of Costa and Friedman \cite{CF1991} implies that
\begin{equation*}
    A_{4,2}>2\sqrt{3} \lofs \approx 0.802,
\end{equation*}
as explained in \textsection\,2.
Throughout this paper, we write $\log^2(x)$ to mean $(\log(x))^2$.
To the best of the authors' knowledge, this was the largest known lower bound for~$A_{4,2}$.

In this paper, we give a larger lower bound for~$A_{4,2}$.
\begin{theorem} \label{main}
    Let~$L$ be a real quartic Galois extension of~$\Q$. Then 
        \begin{equation*} 
            \norm{w}{1} \ge 2\sqrt{6}\lofs\approx 1.134
        \end{equation*} 
    for any nonzero~$w \in \bigwedge^2\LOG(\OLU)$.
    In other words, 
    \begin{equation*}
        A_{4,2} \ge 2\sqrt{6}\lofs\approx 1.134.
    \end{equation*}
\end{theorem}
As demonstrated in \cite{CFS2022}, the Bertrand--Rodriguez Villegas conjecture is equivalent to Lehmer's conjecture \cite{Le1933} when~$j=1$, and to Zimmert's theorem on regulators \cite{Zi1981} when~$j=\operatorname{rank}(\OLU)$.
Hence the case of ~$n=4, j=2$ is neither covered by Lehmer's conjecture nor Zimmert's theorem. 
While our result only covers this specific case, we obtained a better lower bound, which is~$\sqrt{2}$ times the previous known lower bound. 
In fact, for~$L=\Q(\sqrt{2},\sqrt{5})$ and~$w=\pm \LOG(\tfrac{1+\sqrt{5}}{2}) \wedge \LOG(1+\sqrt{2})$,
we have ~$\norm{w}{1}=8\log(\tfrac{1+\sqrt{5}}{2})\log(1+\sqrt{2})$
and thus
\begin{equation*}
    A_{4,2} \leq 8\lof\log(1+\sqrt{2}) \approx 3.393.
\end{equation*}
Hence our lower bound is no less than a third of an upper bound for~$A_{4,2}$.

The proof of the theorem is done separately for the case when ~$\Gal(L/\Q)$ is a Klein four-group, and the case when ~$\Gal(L/\Q)$ is cyclic.
For the Klein four-group case,
we define a subgroup~$E$ generated by the units of the quadratic subfields of~$L$.
Using the Galois module structure of~$E$, we give a lower bound for ~$\norm{w}{1}$ for nonzero element ~$w$ of~$\bigwedge^2\LOG(E)$ and extend this result to~$\OLU$.
For the cyclic group case, we use the structure of~~$\OLU$ given by Hasse \cite{Ha1948} to obtain a basis of~$\bigwedge^2\LOG(\OLU)$.
We then use the result of Pohst \cite{Po1978} to compute the lower bound.

In \textsection\,2 we derive the best lower bound for ~$A_{4,2}$ known so far, from the work of  Costa and Friedman \cite{CF1991}.
In the remaining sections, we prove the main theorem to give the better lower bound for $A_{4,2}$.
The case when ~$\Gal(L/\Q)$ is isomorphic to the Klein four-group is covered in \textsection\,3, and the case when ~$\Gal(L/\Q)$ is isomorphic to the cyclic group is covered in \textsection\,4.

\subsection*{Acknowledgement}
The work of S.S. was supported by the College of Natural Sciences Undergraduate Internship Program from Seoul National University.
D.K. is supported by the National Research Foundation of Korea (NRF)\footnote{the grants No.\,RS-2023-00301976 and No.\,2020R1C1C1A0100681913}.
We also thank the referee for corrections, including those concerning the cyclic case.


\section{The lower bound due to Costa and Friedman} \label{prev}

In this section, we derive the largest lower bound for ~$A_{4,2}$ known so far, to the best of the authors' knowledge.
Using Pohst's result \cite{Po1978}, Costa and Friedman \cite{CF1991} showed that for independent elements~$\epsilon_1,...,\epsilon_j$ of~$\OLU$, the inequality
\begin{equation*} 
    \norm{\LOG(\epsilon_1)\wedge \cdots \wedge \LOG(\epsilon_j)}{2}
    >\left(\frac{n}{\gamma_j}\lofs \right)^{j/2}
\end{equation*}
holds for~$1\leq j<n$, where~$n=[L:\Q]$ and ~$\gamma_j$ is Hermite's constant in dimension~$j$. 
If~$j=n-2=\operatorname{rank}(\OLU)-1$, every element of ~$\bigwedge^{n-2}\LOG(\OLU)$ can be written as ~$d \cdot \LOG(\epsilon_1)\wedge \cdots \wedge \LOG(\epsilon_{n-2})$ where~$d \in \Z$ and~$\LOG(\epsilon_1),...,\LOG(\epsilon_{n-1})$ are the basis of~$\LOG(\OLU)$,
as shown in Lemma 28 of \cite{CFS2022}. 
Therefore, for any nonzero element ~$w$ of ~$\bigwedge^{n-2}\LOG(\OLU)$, the inequality
\begin{equation*}
    \norm{w}{1}\geq\norm{w}{2}
    >\left(\frac{n}{\gamma_j}\lofs \right)^{j/2}
\end{equation*}
holds. 
Restricting to the Galois extensions of~$\Q$, we have
\begin{equation*}
    A_{n,n-2}>\left(\frac{n}{\gamma_{n-2}}\lofs \right)^{\frac{n-2}{2}}
\end{equation*}
and in particular, since~$\gamma_2=\frac{2}{\sqrt{3}}$, we obtain
\begin{equation*}
    A_{4,2}>2\sqrt{3} \lofs \approx 0.802.
\end{equation*}


\section{The Klein four-group Case} \label{klein}
In this section, we prove the main theorem for the case when ~$\Gal(L/\Q)$ is isomorphic to the Klein four-group.
Throughout this section, let~$L$ be a real Galois extension of~$\Q$ with~$\Gal(L/\Q) \cong \Z /2\Z \times \Z /2\Z$.
Let~$\Q(\sqrt{d_1}),\Q(\sqrt{d_2}),\Q(\sqrt{d_3})$ be three quadratic subfields of~$L$, where~$d_1,d_2,d_3$ are square-free integers.
Then we can write~$L=\Q(\sqrt{d_1},\sqrt{d_2})$.
For~$i=1,2,3$, let~$\Gal(L/\Q(\sqrt{d_i}))=\{1,\sigma_i \}$.
Let ~$u_i>1$ be the fundamental unit of~$\Q(\sqrt{d_i})$.
Without loss of generality, let~$1<u_1<u_2<u_3$.
If~$i$ and~$j$ are distinct integers in~$\{1,2,3\}$, since~$\sqrt{d_j} \not\in \Q(\sqrt{d_i})$,~$\sigma_i\in \Gal(L/\Q(\sqrt{d_i}))$ sends~$\sqrt{d_j}$ to its conjugate~$-\sqrt{d_j}$. Thus
\begin{equation} \label{norm}
    u_j\sigma_i(u_j)=\operatorname{N}_{\Q(\sqrt{d_j})/\Q}(u_j)=\pm 1.
\end{equation}
Let~$E=\{u_1^{m_1}u_2^{m_2}u_3^{m_3} | m_1,m_2,m_3 \in \mathbb{Z} \}$.
Then we have the following lemmas, which are immediate consequences of Kubota’s classification of fundamental systems of units in real biquadratic fields \cite[Satz 1]{Ku1956}.

\begin{lemma} \label{indp}
   ~$u_1,u_2,u_3$ are multiplicatively independent. 
    In other words, if integers~$m_1,m_2,m_3$ satisfy~$u_1^{m_1}u_2^{m_2}u_3^{m_3}=1$, then~$m_1=m_2=m_3=0$.
\end{lemma}

\begin{lemma} \label{ind1}
    If~$u\in \OLU$, then~$u^2 \in E$.
\end{lemma}

We first compute the 1-norm of the elements of~$\bigwedge^2\LOG(E)$. 
$\LOG(E)$ is generated by~$\LOG(u_i)$'s. 
Using~$\eqref{norm}$, we can compute~$\LOG(u_i)$'s in terms of~$u_i$'s as follows:
\begin{align*}
    &\LOG(u_1)=(\log(u_1),\ \ \log(u_1),-\log(u_1),-\log(u_1)) \\ 
    &\LOG(u_2)=(\log(u_2),-\log(u_2),\ \ \log(u_2),-\log(u_2)) \nonumber \\
    &\LOG(u_3)=(\log(u_3),-\log(u_3),-\log(u_3),\ \ \log(u_3)) .\nonumber
\end{align*}
Here, the order of the basis of~$\R^{\mathcal{A}_L}$ is~$\delta^{id},\delta^{\sigma_1},\delta^{\sigma_2},\delta^{\sigma_3}$. 
Then the wedge product of these~$\LOG(u_i)$'s are:
\begin{align*}
    &\LOG(u_2)\wedge\LOG(u_3)=
    (\ \ \ \ \ 0,\ \ \ \ \ 0,\ \ 2X_1,\ \ 2X_1,-2X_1,-2X_1) \\
    &\LOG(u_1)\wedge\LOG(u_3)=
    (-2X_2,-2X_2,\ \ 2X_2,-2X_2,\ \ \ \ \ 0,\ \ \ \ \ 0) \\
    &\LOG(u_1)\wedge\LOG(u_2)=
    (-2X_3,\ \ 2X_3,\ \ \ \ \ 0,\ \ \ \ \ 0,\ \ 2X_3,-2X_3) 
\end{align*}
where~$X_1=\log(u_2)\log(u_3), X_2=\log(u_1)\log(u_3),
X_3=\log(u_1)\log(u_2)$ and the order of the basis of~$\bigwedge^2\R^{\mathcal{A}_L}$ is 
$\delta^{id}\wedge\delta^{\sigma_1}, \delta^{\sigma_2}\wedge\delta^{\sigma_3},
\delta^{id}\wedge\delta^{\sigma_3}, \delta^{\sigma_1}\wedge\delta^{\sigma_2},
\delta^{id}\wedge\delta^{\sigma_2},\delta^{\sigma_1}\wedge\delta^{\sigma_3}$. 
Note that~$X_1 > X_2 > X_3 >0$. 

For~$w \in \bigwedge^2 \LOG(E)$,~$w$ can be written as 
\begin{equation*}
    w=n_1\cdot \LOG(u_2)\wedge\LOG(u_3)+n_2 \cdot \LOG(u_1)\wedge\LOG(u_3) + n_3 \cdot  \LOG(u_1)\wedge\LOG(u_2)
\end{equation*}
where~$n_i$'s are integers. Its 1-norm is 
\begin{align*}
    \norm{w}{1} =&\, 2(\abs{n_2 X_2+n_3 X_3}+\abs{n_2 X_2 - n_3 X_3}+\abs{n_1 X_1+n_2 X_2}+\abs{n_1 X_1 - n_2 X_2} \\
     &+\abs{n_1 X_1+n_3 X_3}+\abs{n_1 X_1 - n_3 X_3}) \\
    =&\, 4(\max\{\abs{n_2X_2},\abs{n_3X_3}\}+\max\{\abs{n_1X_1},\abs{n_2X_2}\}+\max\{\abs{n_1X_1},\abs{n_3X_3}\})
\end{align*}
where the second equality is from the following equation.
\begin{equation} \label{summax}
    \abs{X+Y}+\abs{X-Y}=2\max\{\abs{X},\abs{Y}\}\quad \forall X,Y\in\R.
\end{equation}
We have 
\begin{equation*}
    \max\{\abs{n_1X_1},\abs{n_2X_2}\}+\max\{\abs{n_1X_1},\abs{n_3X_3}\} \geq 2\abs{n_1X_1}
\end{equation*}
and thus 
\begin{equation*}
    4(\max\{\abs{n_2X_2},\abs{n_3X_3}\}+\max\{\abs{n_1X_1},\abs{n_2X_2}\}+\max\{\abs{n_1X_1},\abs{n_3X_3}\}) \geq 8\abs{n_1X_1}
\end{equation*}
holds. Similarly, the above inequality also holds when the right hand side is~$8\abs{n_2X_2}$ or~$8\abs{n_3X_3}$.
Therefore
\begin{align*}
    \norm{w}{1}& = 4(\max\{\abs{n_2X_2},\abs{n_3X_3}\}+\max\{\abs{n_1X_1},\abs{n_2X_2}\}+\max\{\abs{n_1X_1},\abs{n_3X_3}\})\\
    &\geq 8\max\{\abs{n_1X_1},\abs{n_2X_2},\abs{n_3X_3} \} 
\end{align*}
and if~$(n_1,n_2,n_3)\neq 0$, then
\begin{equation} \label{fineq}
    \norm{w}{1}\geq 8\abs{X_3}=8\log(u_1)\log(u_2).
\end{equation}

For the case ~$L=\Q(\sqrt{2},\sqrt{5})$, we have ~$u_1=\frac{1+\sqrt{5}}{2}$, ~$u_2=1+\sqrt{2}$, and ~$u_3=3+\sqrt{10}$.
A simple calculation shows that ~$\sqrt{u_1u_2u_3}$ is in ~$\OLU$, and it is easy to verify that ~$u_1,u_2,\sqrt{u_1u_2u_3}$ generate ~$\OLU$.
Then the group ~$\bigwedge^2\LOG(\OLU)$ is generated by~$\LOG(u_1)\wedge \LOG(u_2)$,~$\frac{1}{2}(\LOG(u_1)\wedge\LOG(u_2)+\LOG(u_1)\wedge\LOG(u_3))$ and~$\frac{1}{2}(\LOG(u_1)\wedge\LOG(u_2)+\LOG(u_2)\wedge\LOG(u_3))$.
Thus for any element~$w$ of~$\bigwedge^2\LOG(\OLU)$,~$2w \in \bigwedge^2\LOG(E)$.
Hence we have
\begin{equation} \label{25 case}
    \norm{w}{1}=\frac{1}{2}\norm{2w}{1}\geq 4\lof\log(1+\sqrt{2})\approx 1.697.
\end{equation}

For the case~$L=\Q(\sqrt{5},\sqrt{13})$,~$u_1=\frac{1+\sqrt{5}}{2}, u_2=\frac{3+\sqrt{13}}{2}, u_3=8+\sqrt{65}$.
Similarly,~$u_1,u_2,\sqrt{u_1u_2u_3}$ generate ~$\OLU$ and therefore 
\begin{equation} \label{513 case}
    \norm{w}{1}=\frac{1}{2}\norm{2w}{1}\geq 4\lof\log\left(\tfrac{3+\sqrt{13}}{2}\right)\approx 2.300
\end{equation}
for any nonzero element~$w$ of~$\bigwedge^2\LOG(\OLU)$.

Now consider the case where~$L\neq \Q(\sqrt{2},\sqrt{5})$ and~$L\neq \Q(\sqrt{5},\sqrt{13})$.
By Lemma~\ref{ind1}, for any~$u\in\OLU$,~$u^2\in E$ and thus~$2\LOG(u) \in \LOG(E)$.
Hence if nonzero~$w$ is in~$\bigwedge^2 \LOG(\OLU)$, then~$4w \in \bigwedge^2 \LOG(E)$.
We have
\begin{equation} \label{thin1}
    \norm{w}{1}=\frac{1}{4}\abs{\abs{4w}}_1 \geq 2\log(u_1)\log(u_2)
\end{equation}
by \eqref{fineq}.
Excluding~$\Q(\sqrt{2},\sqrt{5})$ and~$\Q(\sqrt{5},\sqrt{13})$, 
$\Q(\sqrt{3},\sqrt{5})$ has the smallest possible value of ~$\log(u_1)\log(u_2)$,
which is ~$\lof\log(2+\sqrt{3})$.
This is from the Lemma~\ref{lwbd} below and the fact that~$\lof\log(2+\sqrt{3})<\log(1+\sqrt{2})\log\left(\frac{3+\sqrt{13}}{2}\right)$.
Hence from \eqref{thin1}, we see that
\begin{equation*}
    \norm{w}{1} \geq 2\lof\log(2+\sqrt{3})\approx 1.267.
\end{equation*}
With \eqref{25 case} and \eqref{513 case}, we conclude that if~$\Gal(L/\Q)$ is Klein four-group,
\begin{equation*}
    \norm{w}{1}\geq 2\lof\log(2+\sqrt{3}) >2\sqrt{6}\lof^2
\end{equation*}
for any nonzero~$w \in \bigwedge^2\LOG(\OLU)$.

\begin{lemma} \label{lwbd}
    For square-free integer~$m$, let~$v_m>1$ be the fundamental unit of~$\Q(\sqrt{m})$. 
    Then ~$\tfrac{1+\sqrt{5}}{2},1+\sqrt{2},\tfrac{3+\sqrt{13}}{2},2+\sqrt{3}$ are the four smallest possible values of~$v_m$.
\end{lemma}
\begin{proof} 
    First, consider the case~$m \equiv 2,3 \pmod 4$. 
    The fundamental unit~$v_m$ is of the form~$a+b\sqrt{m}$ where
   ~$a,b$ are the smallest positive integers satisfying~$a^2 - mb^2=\pm 1$.
    If~$m\geq6$, then 
    \begin{equation*}
        v_m\geq \sqrt{mb^2-1}+b\sqrt{m}\geq\sqrt{m-1}+\sqrt{m}\geq \sqrt{5}+\sqrt{6}>2+\sqrt{3}.
    \end{equation*}
    Now consider the case~$m \equiv 1 \pmod{4}$.
    Then~$v_m$ is of the form~$\frac{a+b\sqrt{m}}{2}$ where
   ~$a,b$ are the smallest positive integers satisfying~$a^2 - mb^2=\pm 4$.
    If~$m\geq 17$,
    \begin{equation*}
        v_m\geq \frac{\sqrt{mb^2-4}+b\sqrt{m}}{2} \geq \frac{\sqrt{m-4}+\sqrt{m}}{2}\geq  \frac{\sqrt{13}+\sqrt{17}}{2}>2+\sqrt{3}.
    \end{equation*}
    Hence~$v_m\leq2+\sqrt{3}$ only for~$m=2,3,5,13$.
    For those~$m$, we have ~$v_5<v_2<v_{13}<v_3$.
\end{proof}


\section{The Cyclic Group Case} \label{cyclic}

In this section, we prove the main theorem for the case when ~$\Gal(L/\Q)$ is isomorphic to the cyclic group.
Throughout this section, let~$L$ be a real Galois extension of~$\Q$ with~$\Gal(L/\Q) \cong \Z /4\Z$. 
Denote the Galois group as ~$G$.
Let~$\Gal(L/\Q)$ be generated by~$\sigma$. Let~$l$ be the unique quadratic subfield of~$L$.
Let~$u_l>1$ be the fundamental unit of~$l$. Then~$\Olu=\{\pm u_{l}^m \ |\  m \in \Z\}$. First we define relative units as in \cite{Ha1948}.

\begin{definition}
    If~$w \in \OLU$ satisfies~$N_{L/l}(w)=\pm 1$, we say that~$w$ is a \emph{relative unit of~$L$}. 
\end{definition}

Let~$E_{L/l}$ be a group consisting of all relative units of~$L$. 
As in \cite{Gr1979}, let~$E^L$ be a sub~$G$-module of~$\OLU$ generated by elements of ~$\Olu$ and~$E_{L/l}$. 
Let~$Q=[\OLU:E^L]$.
The following two lemmas are due to Hasse \cite[Satz 16 and 22]{Ha1948}.

\begin{lemma} \label{ind2}
    The index~$Q$ is either 1 or 2. Furthermore, ~$Q=2$ if and only if there exists~$u_*\in\OLU$ such that~$N_{L/l}(u_*)=\pm u_l$.
\end{lemma}

\begin{lemma} \label{hassest}
    There exists an element~$u_0$ of~$E_{L/l}$ such that~$-1,u_0$ and its conjugate~$\sigma(u_0)$ generates~$E_{L/l}$.
\end{lemma}

Define $u_0$ as in the lemma.
We have
\begin{equation*}
    E^L=\{\pm u_l^{m_1}u_0^{m_2}\sigma(u_0)^{m_3} \ | \ (m_1,m_2,m_3) \in \Z^3 \}.
\end{equation*}
We first compute the 1-norm of the elements of~$\bigwedge^2 \LOG(E^L)$. Note that since~$u_0$ is relative unit of~$L$, 
$\sigma^2(u_0)=\pm \frac{1}{u_0}$.
We can compute~$\LOG$'s as follows:
\begin{align} 
\begin{split}\label{LOGS}
    \LOG(u_l) = (W_1,-W_1,W_1,-W_1) \\ 
    \LOG(u_0) = (W_2,W_3,-W_2,-W_3) \\
    \LOG(\sigma(u_0)) = (W_3,-W_2,-W_3,W_2)     
\end{split}
\end{align}
where~$W_1=\log(u_l),W_2=\log(\abs{u_0}),W_3=\log(\abs{\sigma(u_0)})$.
Here, the order of the basis of~$\R^{\mathcal{A}_L}$ is 
$\delta^{id},\delta^{\sigma},\delta^{\sigma^2},\delta^{\sigma^3}$. 
Note that~$W_1,W_2,W_3$ are nonzero since~$u_l,u_0,\sigma(u_0)$ are not~$\pm1$.
Then the wedge product of these terms are:
\begin{align*}
    \LOG(u_l)\wedge\LOG(u_0)=
    & (\ \ Y_4,-Y_4,Y_5,\ \ Y_5,-Y_2,\ \ Y_3) \\
    \LOG(u_l)\wedge\LOG(\sigma(u_0))=
    & (-Y_5,\ \ Y_5,Y_4,\ \ Y_4,-Y_3,-Y_2) \\
    \LOG(u_0)\wedge\LOG(\sigma(u_0))=
    & (-Y_1,-Y_1,Y_1,-Y_1,\ \ \ \ 0,\ \ \ \ 0) 
\end{align*}
where~$Y_1=W_2^2+W_3^2,Y_2=2W_1W_2,Y_3=2W_1W_3,Y_4=W_1W_2+W_1W_3, 
Y_5=W_1W_2-W_1W_3$. 
Here the order of the basis of~$\bigwedge^2\R^{\mathcal{A}_L}$ is 
$\delta^{id}\wedge\delta^{\sigma}, 
\delta^{\sigma^2}\wedge\delta^{\sigma^3},
\delta^{id}\wedge\delta^{\sigma^3}, \delta^{\sigma}\wedge\delta^{\sigma^2},
\delta^{id}\wedge\delta^{\sigma^2},\delta^{\sigma}\wedge\delta^{\sigma^3}$.

For any nonzero element ~$w \in \bigwedge^2 \LOG(E^L)$, it can be written as 
\begin{align*}
    w=&\, n_1\cdot \LOG(u_l)\wedge\LOG(u_0)+n_2 \cdot \LOG(u_l)\wedge\LOG(\sigma(u_0)) \\
    &+ n_3 \cdot  \LOG(u_0)\wedge\LOG(\sigma(u_0))
\end{align*}
where~$n_1,n_2,n_3$ are integers.
Then~$\norm{w}{1}=f(n_1,n_2,n_3)$ where the map~$f:\Z^3 \rightarrow \R$ is given as
\begin{align*}
    f(n_1,n_2,n_3)=&\, \abs{n_1Y_4-n_2Y_5-n_3Y_1}+\abs{-n_1Y_4+n_2Y_5-n_3Y_1}\\
    &+\abs{n_1Y_5+n_2Y_4+n_3Y_1}+\abs{n_1Y_5+n_2Y_4-n_3Y_1}\\
    &+\abs{-n_1Y_2-n_2Y_3}+\abs{n_1Y_3-n_2Y_2}.   
\end{align*}
Applying \eqref{summax}, above equation becomes
\begin{align*}
    f(n_1,n_2,n_3)=&2\max\{\abs{n_1Y_4-n_2Y_5},\abs{n_3Y_1}\} 
    +2\max\{\abs{n_1Y_5+n_2Y_4},\abs{n_3Y_1}\}\\
    &+\abs{-n_1Y_2-n_2Y_3}+\abs{n_1Y_3-n_2Y_2}.   
\end{align*}
In particular, we have 
\begin{align} \label{n3case}
    f(n_1,n_2,n_3)\geq 4\abs{n_3Y_1}=4\abs{n_3}(W_2^2+W_3^2)
\end{align}
and
\begin{align*}
    f(n_1,n_2,n_3)\geq&\, 2\abs{n_1Y_4-n_2Y_5}+2\abs{n_1Y_5+n_2Y_4}\\
    &+\abs{-n_1Y_2-n_2Y_3}+\abs{n_1Y_3-n_2Y_2}.
\end{align*}
If~$(n_1,n_2)\neq(0,0)$, then by Lemma~\ref{absin} below,
\begin{align*}
    f(n_1,n_2,n_3)&\geq 2\abs{Y_4}+2\abs{Y_5}+\abs{Y_2}+\abs{Y_3}\\
    &=2\abs{W_1W_2+W_1W_3}+2\abs{W_1W_2-W_1W_3}+2\abs{W_1W_2}+2\abs{W_1W_3}
\end{align*}
and applying \eqref{summax} once more, we have
\begin{align} \label{n12case}
    f(n_1,n_2,n_3)\geq 2W_1(2\max\{\abs{W_2},\abs{W_3}\}+\abs{W_2}+\abs{W_3})
\end{align}
Using these results, we prove Theorem~\ref{main} for each case~$Q=1$ and~$Q=2$.

\begin{lemma} \label{absin}
    Let~$X, Y \in \R$. For any~$(m,n) \in \Z^2 \backslash \{(0,0)\}$, we have
    \begin{equation*}
        \abs{mX+nY}+\abs{nX-mY} \geq \abs{X}+\abs{Y}.
    \end{equation*}
\end{lemma}
\begin{proof}
    The inequality is trivial if~$X=0$ or~$Y=0$. 
    Assume that they are both nonzero.
    We may further assume that they are positive and~$X\geq Y$.
    Let~$Z:=Y/X\leq 1$.
    Since
    \begin{equation*}
        \abs{mX+nY}+\abs{nX-mY}=\abs{X}(\abs{m+nZ}+\abs{n-mZ}),
    \end{equation*}
    we only need to show the inequality
    \begin{equation*}
        \abs{m+nZ}+\abs{n-mZ}\geq 1+Z.
    \end{equation*}
    Without loss of generality, assume~$m\geq 0$.
    If~$m=0$,~$\abs{nZ}+\abs{n}\geq 1+Z$ for any nonzero integer~$n$.
    If~$m\geq 2$, 
    \begin{align*}
    \abs{m+nZ}+\abs{n-mZ}&\geq\abs{m+nZ}+\abs{Z(n-mZ)}
        \;\geq\abs{(m+nZ)-Z(n-mZ)} \\
        &=\abs{m(1+Z^2)}\;\geq m \;\geq 2 \;\geq 1+Z.
    \end{align*}
    Now consider the case~$m=1$. The case~$n=0$ is trivial. If~$n\leq-1$, then
    \begin{equation*}
        \abs{1+nZ}+\abs{n-Z}\geq{-n+Z}\geq1+Z,
    \end{equation*}
    and if~$n \geq 1$,
    \begin{equation*}
        \abs{1+nZ}+\abs{n-Z}\geq{1+nZ}\geq1+Z.
    \end{equation*}    
\end{proof}

\subsection{The Case Q=1}
Since~$\OLU=E^L$, for any nonzero~$w\in\bigwedge^2\LOG(\OLU)$, its norm is of the form ~$\norm{w}{1}=f(n_1,n_2,n_3)$ where~$(n_1,n_2,n_3)\neq(0,0,0)$.
If~$n_3\neq0$, then \eqref{n3case} gives
\begin{equation*}
    \norm{w}{1}\geq 4(W_2^2+W_3^2)
\end{equation*}
and if~$n_3=0$, then~$(n_1,n_2)\neq(0,0)$, so \eqref{n12case} gives
\begin{align*}
    \norm{w}{1}\geq 2W_1(2\max\{\abs{W_2},\abs{W_3}\}+\abs{W_2}+\abs{W_3}).
\end{align*}
Therefore we have
\begin{equation} \label{q1final}
    \norm{w}{1}\geq \min\{4(W_2^2+W_3^2), 2W_1(2\max\{\abs{W_2},\abs{W_3}\}+\abs{W_2}+\abs{W_3}) \}
\end{equation}
for any nonzero~$w\in\bigwedge\LOG(\OLU)$.

We use the following theorem from Pohst \cite{Po1978} to give a lower bound for \eqref{q1final}.
\begin{theorem} \label{PO}
    Let~$L$ be a totally real number field and~$u\in\OLU$ with~$u\neq\pm1$. Then we have
    \begin{equation*}
        \norm{\LOG(u)}{2}^2\geq[L:\Q]\lofs.
    \end{equation*} 
\end{theorem}
\begin{proof}
    Bemerkung in p. 102 of Pohst \cite{Po1978} states that the Lemma in p. 98 applies for ~$M$ of the equation (10) in p. 97. Then the inequality (11) of the Lemma gives the required result.
\end{proof}
Then from the above theorem and \eqref{LOGS}, 
\begin{equation} \label{nonfund}
    W_2^2+W_3^2=\frac{1}{2}\norm{\LOG(u_0)}{2}^2\geq2\lofs.
\end{equation}
Under this constraint, the following elementary fact will be useful.
\begin{lemma} \label{cal}
    Let $x,y \in \R$ with $x^2+y^2\ge R^2$ for some positive real number $R$.
    Then,
    \begin{equation*}
        2\max\{\abs{x},\abs{y}\}+\abs{x}+\abs{y}\ge 2\sqrt{2}R.
    \end{equation*}
\end{lemma}

\begin{proof}
    Without loss of generality, assume $x\ge y\ge0$ and $x^2+y^2=R^2$.
    Then there exists $\theta\in [0,\tfrac{\pi}{4}]$ such that
    $x=R\cos\theta$ and $y=R\sin\theta$.
    Hence
    \begin{equation*}
        2\max\{\abs{x},\abs{y}\}+\abs{x}+\abs{y}=3x+y=\sqrt{10}R\sin(\theta+\phi),
    \end{equation*}
    where $\phi=\arctan(3)$.
    Since $\theta+\phi$ ranges over $[\arctan(3),\arctan(3)+\tfrac{\pi}{4}]\subset (0,\pi)$ and
    $\sin(\phi)=\tfrac{3}{\sqrt{10}}$ is larger than $\sin(\phi+\tfrac{\pi}{4})=\tfrac{2}{\sqrt{5}}$,
    the minimum of $\sin(\theta+\phi)$ occurs at $\theta=\tfrac{\pi}{4}$.
    The minimum value is therefore
        $\sqrt{10}\,R \cdot \tfrac{2}{\sqrt{5}} \;=\; 2\sqrt{2}\,R$.
\end{proof}

Combining Lemma~\ref{cal} with \eqref{nonfund}, we obtain
\begin{equation*}
    2\max\{\abs{W_2},\abs{W_3}\}+\abs{W_2}+\abs{W_3}\ge4\lof.
\end{equation*}
Since~$u_l$ is a fundamental unit of quadratic field, by Lemma~\ref{lwbd},
\begin{equation} \label{fund}
    W_1=\log(u_l)\geq\lof
\end{equation}
holds, and with the previous inequality,
\begin{equation*}
    2W_1(2\max\{\abs{W_2},\abs{W_3}\}+\abs{W_2}+\abs{W_3})\ge8\lofs.
\end{equation*}
With \eqref{nonfund}, the inequality \eqref{q1final} becomes
\begin{equation}\label{q1bound}
      \norm{w}{1}\geq8\lofs
\end{equation}
for any nonzero~$w\in\bigwedge^2\LOG(\OLU)$.
Thus Theorem~\ref{main} holds for the case~$Q=1$.


\subsection{The Case Q=2}

We use the following lemma due to Hasse \cite[Satz 28]{Ha1948}.

\begin{lemma} \label{u*}
If $Q=2$, there exists a unique unit $u_*>0$ of $L$ such that
\[
|N_{L/l}(u_*)| = u_l, \qquad |u_* \sigma(u_*)| = \abs{u_0}.
\]
\end{lemma}

Define~$u_*$ as in the lemma. Then
\begin{equation} \label{ust}
    u_*^2=\frac{u_*\sigma^2(u_*)\cdot u_*\sigma(u_*)}{\sigma(u_*\sigma(u_*))}
    =\pm\frac{u_lu_0}{\sigma(u_0)}
\end{equation}
and therefore ~$\sigma(u_0)=\pm u_lu_0u_*^{-2}$.
The proof of Lemma~\ref{ind2} shows that when~$Q=2$, every unit~$u\in\OLU$ can be written uniquely as the form~$\pm u_*^{n_0}u_l^{n_1}u_0^{n_2}\sigma(u_0)^{n_3}$ where~$n_0$ is~$0$ or~$1$ and~$n_1,n_2,n_3$ are integers.
Thus ~$\OLU$ is generated by~$-1,u_l,u_0,u_*$.
Hence ~$\bigwedge^2\LOG(\OLU)$ is generated by~$\LOG(u_l)\wedge\LOG(u_0), \LOG(u_l)\wedge\LOG(u_*)$ and ~$\LOG(u_0)\wedge\LOG(u_*)$.
From \eqref{ust}, we also have 
\begin{equation} \label{LOG*}
    \LOG(u_*)=\frac{1}{2}(\LOG(u_l)+\LOG(u_0)-\LOG(\sigma(u_0)))
\end{equation}
and thus
\begin{align*}
\LOG(u_l)\wedge\LOG(u_*)&=\frac{1}{2}(\LOG(u_l)\wedge\LOG(u_0)-\LOG(u_l)\wedge\LOG(\sigma(u_0))),\\
\LOG(u_0)\wedge\LOG(u_*)&=-\frac{1}{2}(\LOG(u_l)\wedge\LOG(u_0)+\LOG(u_0)\wedge\LOG(\sigma(u_0))).
\end{align*}
Hence for any~$w\in\bigwedge^2\LOG(\OLU)$, 
\begin{align*}
    2w&=n_1\cdot \LOG(u_l)\wedge\LOG(u_0)+n_2 \cdot \LOG(u_l)\wedge\LOG(\sigma(u_0)) \\
    &+ n_3 \cdot  \LOG(u_0)\wedge\LOG(\sigma(u_0))
\end{align*}
where~$n_1,n_2,n_3$ are integers and~$n_1+n_2+n_3$ is even
and thus 
\begin{equation} \label{halff}
    \norm{w}{1}=\frac{1}{2}f(n_1,n_2,n_3).
\end{equation}

We now proceed to give a lower bound for~$f(n_1,n_2,n_3)$ where~$n_1+n_2+n_3$ is even and~$(n_1,n_2,n_3)\neq(0,0,0)$.
If~$(n_1,n_2)=(0,0)$, then~$\abs{n_3}\geq 2$, so \eqref{n3case} gives
\begin{equation*}
    f(n_1,n_2,n_3) \geq 8(W_2^2+W_3^2).
\end{equation*}
Since \eqref{nonfund} still holds in the~$Q=2$ case, we have
\begin{equation} \label{n3q2}
    f(n_1,n_2,n_3) \geq 16\lofs.
\end{equation}
Note that \eqref{fund} also still holds.

Now consider the case~$(n_1,n_2)\neq(0,0)$.
Since~$\LOG(u_l),\LOG(u_0), \LOG(\sigma(u_0))~$ are perpendicular as follows from \eqref{LOGS}, we have from \eqref{LOG*}
\begin{align*}
    \norm{\LOG(u_*)}{2}^2&=\frac{1}{4}(\norm{\LOG(u_l)}{2}^2+\norm{\LOG(u_0)}{2}^2+\norm{\LOG(\sigma(u_0))}{2}^2)\\
    &=W_1^2+W_2^2+W_3^2.
\end{align*}
Then Theorem~\ref{PO} gives
\begin{equation*}
    W_1^2+W_2^2+W_3^2\geq4\lofs.
\end{equation*}
Letting $S=\sqrt{W_2^2+W_3^2}$, the inequality becomes
\begin{equation} \label{circ}
    W_1^2+S^2\ge 4\lofs,
\end{equation}
while \eqref{nonfund} gives $S\ge \sqrt{2}\lof$.

We claim that 
\begin{equation}\label{WS}
    W_1S\ge\sqrt{3}\lofs.
\end{equation}
Indeed, if $W_1\ge \sqrt{2}\lof$, then since $S\ge \sqrt{2}\lof$, 
\begin{equation*}
    W_1S\ge 2\lofs>\sqrt{3}\lofs.
\end{equation*}
Otherwise, by\eqref{fund}, we have $W_1\in\left[\lof,\sqrt{2}\lof\right)$.
Then
\begin{equation*}
    W_1S\ge\sqrt{W_1^2(4\lofs-W_1^2)}=\sqrt{4\log^4\left(\tfrac{1+\sqrt{5}}{2}\right)-\left(W_1^2-2\lofs\right)^2}.
\end{equation*}
This expression attains its minimum at $W_1=\lof$, yielding $\sqrt{3}\lofs$.
This proves the claim.

Combining Lemma ~\ref{cal} with \eqref{WS}, we obtain
\begin{align*}
    2W_1(2\max\{\abs{W_2},\abs{W_3}\}+\abs{W_2}+\abs{W_3})\ge 4\sqrt{2}W_1S\ge 4\sqrt{6}\lofs.
\end{align*}
With \eqref{n12case}, we have 
\begin{equation*}
    f(n_1,n_2,n_3) \ge 4\sqrt{6}\lof^2.
\end{equation*}
Since~$16\lofs>4\sqrt{6}\lofs$, by \eqref{n3q2}, the above inequality holds for any
$(n_1,n_2,n_3)\neq(0,0,0)$ with~$n_1+n_2+n_3$ even.
Therefore, with \eqref{halff},
\begin{equation*} 
    \norm{w}{1}\ge 2\sqrt{6}\lofs
\end{equation*}
for any nonzero~$w\in\bigwedge^2\LOG(\OLU)$.
This concludes the proof of Theorem~\ref{main}.


\end{document}